\documentclass{amsart}
\usepackage{amsmath, amsfonts, amssymb, amsthm, amscd, enumerate}
\usepackage{hyperref}

\newtheorem{theorem}{Theorem}[section]
\newtheorem{lemma}[theorem]{Lemma}
\newtheorem{corollary}[theorem]{Corollary} 

\newtheorem{proposition}[theorem]{Proposition}

\newtheorem{intthm}{Theorem}

\theoremstyle{definition}
\newtheorem{definition}[theorem]{{Definition}}
\newtheorem{example}[theorem]{Example}
\newtheorem{remark}[theorem]{Remark}

\newtheorem{chunk}[theorem]{\hspace*{-.77ex}\bf}
  
\newtheorem{fact}[theorem]{\bf Fact}

\newtheorem*{chunk*}{}

\numberwithin{equation}{theorem}

\input xy
\xyoption{all}

\def\p{{\frak p}}

\def\m{{\frak m}}
\def\H{\operatorname{H}}
\def\Hom{\operatorname{Hom}}

\def\ker{\operatorname{ker}}

\def\Coker{\operatorname{Coker}}

\newcommand{\Spec}{\operatorname{Spec}}
\def\Cl{\operatorname{Cl}}
\def\rad{\operatorname{rad}}

\def\dim{\operatorname{dim}}

\def\nl{\operatorname{nil}}

\def\ann{\operatorname{ann}}

\def\ker{\operatorname{ker}}

\def\Coker{\operatorname{Coker}}

\def\p{\mathfrak{p}}

\newcommand{\ol}{\overline}
\newcommand{\Ht}{\operatorname{ht}}
\newcommand{\pd}{\operatorname{pd}}
\newcommand{\fp}{\mathfrak p}
\newcommand{\n}{\mathfrak n}
\newcommand{\fa}{\mathfrak a}
\newcommand{\fm}{\mathfrak m}

\newcommand{\supp}{\operatorname{Supp}}

\newcommand{\abarhat}{\widehat{\bar A}}

\newcommand{\dual}[2]{#1^{#2}}
\newcommand{\ddual}[2]{#1^{#2#2}}
\newcommand{\pdual}[2]{(#1)^{#2}}
\newcommand{\pddual}[2]{(#1)^{#2#2}}

\newcommand{\bidual}[2]{\sigma^{#1}_{#2}}
\newcommand{\cl}[1]{[#1]}
\newcommand{\fb}{\mathfrak b}
\newcommand{\vf}{\varphi}
\renewcommand{\ker}{\operatorname{Ker}}

\newcommand{\into}{\hookrightarrow}

\begin{document}
\bibliographystyle{amsplain}

\title
{Torsion in kernels of induced maps on divisor class groups}

\author[S.\ Sather-Wagstaff]{Sean Sather-Wagstaff}
\address{Sean Sather-Wagstaff, Department of Mathematics,
NDSU Dept \# 2750,
PO Box 6050,
Fargo, ND 58108-6050
USA}

\email{sean.sather-wagstaff@ndsu.edu}

\urladdr{http://www.ndsu.edu/pubweb/\~{}ssatherw/}

\author{Sandra Spiroff}
\address{Sandra Spiroff, Department of Mathematics, 
305 Hume Hall,
P.~O.~Box 1848,
University of Mississippi, University, MS  38677-1848, USA}
\email{spiroff@olemiss.edu}
\urladdr{http://home.olemiss.edu/\~{ }spiroff/}

\begin{abstract}  We investigate torsion elements in the kernel of the map on divisor class groups of excellent local normal domains $A$ and $A/I$, for an ideal $I$ of finite projective dimension.  The motivation for this work is a result of Griffith-Weston which applies when $I$ is principal.
\end{abstract}



\thanks{Sean Sather-Wagstaff was supported in part by a grant from the NSA}

\maketitle



\section*{Introduction}

In this section, let $A\to B$ be a homomorphism of noetherian normal integral domains.
Under certain circumstances, one can show that such a map 
induces a group homomorphism 
$\Cl(A) \to \Cl(B)$ on divisor class groups given by $\cl\fa\mapsto\cl{(\frak a \otimes_A B)^{BB}}$.
(See Section~\ref{sec01} for definitions and background material.) 
For instance, if $A$ is excellent and local, $t$ is a prime element of $A$ such that $A/tA$ satisfies the regularity condition $(R_1)$,
and $B$ is the integral closure of $A/tA$, then it is well-known that the natural map $A\to B$ induces such a homomorphism. 
In this setting, the kernel of this induced map has been studied by V.I.~Danilov, P.~Griffith, J.~Lipman, C.~Miller,  S.~Spiroff, D.~Weston, and others. 
For instance, Griffith and Weston prove the following:

\begin{intthm}  \label{GW} \cite[Theorem 1.2]{GW} Let $(A, \frak m)$ be an excellent, local, normal domain and let $t$ be a principal prime in $\frak m$ such that $A/tA$ satisfies the condition $(R_1)$.  Let $B$ denote the integral closure of $A/tA$ and let $e >1$ be an integer which represents a unit in $A$.  Then the kernel of the homomorphism $\Cl(A) \to \Cl(B)$ contains no element of order $e$.
\end{intthm}

A careful analysis of their intricate argument reveals an overlooked detail. Specifically, they suppose that $\cl\fa$ is an element of order $e$ in 
$\ker(\Cl(A) \to \Cl(B))$, and let $a\in A$ be such that $\fa^{(e)}=aA$. 
They describe an
$A$-algebra structure on the direct sum $S=\oplus_{i=0}^{e-1}\fa^{(i)}$ that makes $S$ isomorphic to the integral closure of
$A[T]/(T^e-a)$.  They then show that the integral closure of $S/tS$ is \'etale over $B$ and conclude that $e=1$.
However, in their proof of the last step, they assume that the integral closure of $S/tS$ is local. 
While this may be true in general, we are only certain of it when $A/t A$ is normal, by a result of M.~Hochster and C.~Huneke~\cite[(3.9) Proposition c)]{HH},
once one interprets the integral closure of $S/tS$ as the $S_2$-ification of $S/tS$.

Using the point of view of $S_2$-ifications, we obtain the theorem below. Its proof is the content of Section~\ref{sec130110a}.
This result allows us to obtain the special case of Theorem~\ref{GW} where $A/tA$ is normal; see Corollary~\ref{cor0401}.

\begin{intthm}  \label{thmtwo} Let $(A, \frak m)$ be an excellent, local, normal domain and $I$ a prime ideal of $A$ with finite projective dimension such that:
\begin{enumerate}
\item[(i)] $\bar A = A/I$ is normal; 
\item[(ii)] $I$ is a complete intersection on the punctured spectrum of $A$, i.e.,
for each prime ideal $\frak p \neq \frak m$, the localization $I_{\frak p}$ is either equal to $A_{\frak p}$ or generated by a regular sequence in $A_{\fp}$; and
\item[(iii)] $\mu(I) \leq \dim A-2$, where $\mu(I)$ is the minimal number of generators of $I$.
\end{enumerate}
Then for any positive integer $e > 1$ which represents a unit in $A$, the kernel of the homomorphism $\Cl(A) \to \Cl(\bar A)$ contains no elements of order $e$.
\end{intthm}

A significant tool in the proof of Theorem \ref{thmtwo} is the next result, the proof of which is the content of Section~\ref{sec02}.

\begin{intthm}  \label{thmone} Let $(A, \frak m)$ be a complete, local, normal domain such that $A/\frak m$ is separably closed and let $I$ be a prime ideal of $A$ with finite projective dimension such that $\bar A = A/I$ is normal.  Let $e >1$ be an integer which represents a unit in $A$ and assume that $A$ contains a primitive $e$-th root of unity. Let $[\frak a]$ be an element in $\Cl(A)$ with order $e$.   Set $\displaystyle{R = A \oplus \frak a \oplus \frak a^{(2)} \oplus \dots \oplus \frak a^{(e-1)} }$.  If any of the equivalent conditions of Fact \ref{HH} hold for $R/IR$, then $[\frak a] \notin \ker(\Cl(A) \to \Cl(\bar A))$. 
\end{intthm}

Note that the $A$-algebra structure on $R$ is described explicitly in paragraph~\ref{assumptions}, and that Lemma~\ref{barRequidim} shows that $R/IR$ satisfies the
hypotheses of Fact \ref{HH}. Also, the proofs of Theorems~\ref{thmtwo} and~\ref{thmone} use ideas from \cite[Theorem 1.2]{GW}, but with some key changes.  For completeness and clarification, we provide the details.  Because of their length and technicality, the proofs are presented in a series of lemmas.  
We conclude the paper with a few corollaries and an example in Section~\ref{sec130110b}.

Finally, we mention that a thorough discussion of the background material and list of terminology is given Section~\ref{sec01} for the reader's convenience.  We provide useful definitions and results on divisor class groups,  $S_2$-ifications, and unramified/\'etale extensions.
Undefined terms can be found in~\cite{M}.

\section{Background} \label{sec01}

{\it Throughout this paper, we assume that all rings are commutative and noetherian.}


\subsection*{Divisor Class Groups}

\

We begin with our working definition of the divisor class group of a  normal domain.
It  can be found in J. Lipman \cite[\S 0]{L} and is 
equivalent to the classical additive definition of the divisor class 
group appearing in N. Bourbaki \cite[VII \S 1]{B} and R. Fossum \cite[\S 6]{F}.  A  discussion of this equivalence appears in \cite[2.10]{SSW}.  

\begin{definition} \label{defn0101}
Let $A$ be a  normal domain and $M$ a
finitely generated $A$-module.  The \emph{dual} of $M$ is 
$\dual{M}{A}=\Hom_A(M, A)$ and the \emph{double dual} of $M$ is 
$\ddual{M}{A} = \pdual{\dual{M}{A}}{A}$.  The natural \emph{biduality map}
$\bidual{A}{M}\colon M \to \ddual{M}{A}$ is the $A$-module homomorphism 
given by $\bidual{A}{M}(m)(g)=g(m)$  for all $m \in M$ and all $g \in \dual{M}{A}$.  
We say that $M$ is \emph{reflexive} if $\bidual{A}{M}$ is an isomorphism.  
\end{definition}

\begin{remark} \label{rmk0101}  
As in our previous work \cite{SWS}, we use the notation $\dual{M}{A}$ for the dual of an $A$-module $M$ in order to avoid ambiguity when we work with
two or more rings simultaneously.  This replaces the notation $M^*$ from the classical literature.
\end{remark}
 
\begin{definition} \label{defn0102}The \emph{divisor class group} of a normal domain $A$, denoted $\Cl(A)$, is the group 
of isomorphism classes of reflexive $A$-modules of rank one, or equivalently, non-zero reflexive ideals of $A$.  
An element $\cl\fa \in \Cl(A)$ is called a \emph{divisor class}, and multiplication is
defined by $\cl\fa \cdot \cl\fb = \cl{\pddual{\fa \otimes_A \fb}{A}}$.
The identity element is $\cl A$, and the inverse of
$\cl{\fa}$ is $\cl{\dual{\fa}{A}}$.

As $A$ is normal, it satisfies the Serre conditions $(R_1)$ and $(S_2)$.  Thus, each reflexive  ideal $\fa$ can be written uniquely as the primary 
decomposition $\fa=\bigcap_{j=1}^s \fp_j^{(e_j)}$, where the $\fp_j$ are 
the height one prime ideals containing $\fa$.  For each positive integer $d$, we set
$\fa^{(d)}=\bigcap_{j=1}^s \fp_j^{(e_jd)}$. In $\Cl(A)$, this definition implies that $d\cl{\fa}=\cl{\fa^{(d)}}$.
\end{definition}

\begin{fact}
Let $A\to B$ be a  homomorphism of finite flat dimension between normal domains.
(For instance, if $I$ is a prime ideal of $A$ with finite projective dimension such that $A$ and $A/I$ are normal domains,
then we can consider the natural surjection $A\to A/I$.)
Then~\cite[Theorem A]{SSW}  provides a well-defined group homomorphism  $\Cl(A) \to \Cl(B)$ given by $[\frak a] \mapsto [(\frak a \otimes_A B)^{BB}]$. 
\end{fact}

We complement this with the following.

\begin{lemma}
Let $A$ be a normal domain, and let $I$ be a prime ideal of $A$ such that $\bar A=A/I$ is excellent and satisfies $(R_1)$.
If $B$ is the integral closure of $\bar A$, then  the natural ring homomorphism $\vf\colon A\to B$
induces a well-defined group homomorphism  $\Cl(A) \to \Cl(B)$ given by $[\frak a] \mapsto [(\frak a \otimes_A B)^{BB}]$. 
\end{lemma}

\begin{proof}
Note that the fact that $\bar A$ is excellent implies that $B$ is noetherian.
Let $Q$ be a height one prime ideal of $B$. Then $\bar P:=Q\cap \bar A=P/I$ for some prime ideal $P$ of $A$.
By~\cite[Theorem 1.10]{SWS}, it suffices to show that $A_P$ is a unique factorization domain, e.g., that it is regular.
By~\cite[Theorems 4.8.6 and B.5.1]{HS}, we know that $\Ht \bar P=\Ht Q $; this also uses 
the fact that $\bar A$ is excellent.
Since $\bar A$ satisfies $(R_1)$, it follows that $\bar A_{\bar P}$ is regular.
Since the induced map $A_P\to \bar A_{\bar P}$ has finite projective dimension, we conclude from~\cite[Theorem 6.1(1)]{AFHa}
that $A_P$ is regular.
\end{proof}

\subsection*{$S_2$-ifications}

\

We present here an exposition on $S_2$-ifications, taking much of our content from a paper by M.~Hochster and C.~Huneke \cite{HH}.  This topic comes to the fore in the guise of the integral closure when one assumes that the domain in question satisfies the regularity condition $(R_1)$, but is not necessarily normal.  This allows one to consider a broader class of rings when studying divisor class groups.

\begin{definition} 
Let $(A,\m)$ be a local ring, and let $E$ be the injective hull of $A/\frak m$ over $A$.
A \textit{canonical module} for $A$ is a finitely generated $A$-module $\omega$ such that
$\Hom_A(\omega,E)\cong\operatorname{H}^{\dim(A)}_{\m}(A)$.
\end{definition}

\begin{fact}
Let $A$ be a local ring.
If $A$ is a homomorphic image of a Gorenstein local ring, e.g., $A$ is complete, then it has a canonical module.
\end{fact}

\begin{definition} \label{j-ideal} Let $A$ be a local ring.  Denote by $j(A)$ the largest ideal which is a submodule of $A$ of dimension smaller than $\dim A$.  Specifically, 
$$j(A) = \{a \in A : \dim(A/ \ann_A(a)) < \dim A \}.$$
\end{definition}

\begin{fact}  \label{fact1} The following items are from~\cite[(2.2f), (2.1)]{HH}. 
\begin{enumerate}[(a)]
\item If $A$ is local with canonical module $\omega$, then $j(A) = \ker(A \to \Hom_A(\omega, \omega))$.
\item A local ring $A$ is equidimensional and (0) has no embedded primes if and only if $j(A)= (0)$.
In particular, if $A$ is a local domain, then $j(A)=0$.
\end{enumerate}
\end{fact}

\begin{definition}  \label{S2def} \cite[(2.3)]{HH}  
Let $A$ be a local ring.
\begin{enumerate}[(a)]
\item
If $j(A) = 0$, then a subring $S$ of the total  ring of quotients of $A$ is an {\it $S_2$-ification} of $A$ if:
\begin{enumerate}
\item[$\bullet$] $S$ is module finite over $A$;
\item[$\bullet$] $S$ satisfies the Serre condition $(S_2)$ over $A$; and
\item[$\bullet$] $\Coker(A \to S)$ has no prime ideal of $A$ of height less than two in its support.  
\end{enumerate}
\item
When $A$ is equidimensional but possibly $j(A) \neq 0$, then by an {\it $S_2$-ification} of $A$, we mean 
 an $S_2$-ification of $A/j(A).$
\end{enumerate}
\end{definition}

\begin{fact}  \label{fact2} \cite[(2.2), (2.7)]{HH} A local ring $A$ has an {\it $S_2$-ification} if it has a canonical  module $\omega$.
Specifically, $\Hom_A(\omega, \omega)$ is a commutative $A$-algebra that satisfies $(S_2)$
both as a ring and an $A$-module, and it is an $S_2$-ification of $A$.  
\end{fact}

\begin{lemma}  \label{domainfact} Assume that $A$ is a local ring with a canonical module that satisfies $(R_1)$, such that $j(A)=0$.
Then the $S_2$-ification $S$ of $A$ coincides with the integral closure of $A$ in its total ring of quotients. 
\end{lemma}

\begin{proof}
Since the $S_2$-ification of $A$ is unique up to isomorphism, by~\cite[(2.5)]{HH}, it suffices to show that $S$ is integrally closed
in its total ring of fractions. It suffices to show that $S$ satisfies $(R_1)$ and $(S_2)$ as a ring.
Let $P$ be a prime ideal of $S$, and set $\p=P\cap A$.
From~\cite[Proposition 3.5(a)]{HH}, we know that $\Ht \p=\Ht P$.
Hence, the fact that $S$ satisfies $(S_2)$ as a ring follows from the fact that it satisfies $(S_2)$ as an $A$-module, by~\cite[Corollaire (5.7.11)]{Gr4-2}.

To verify that $S$ satisfies $(R_1)$, assume  that $\Ht P \leq 1$.
Since $A$ satisfies $(R_1)$, the ring $A_\p$ is regular. Since $\Coker(A \to S)$ has no prime ideal of $A$ of height less than two in its support,
we conclude that $A_\p\cong S_\p$, hence $A_\p\cong S_P$ is regular.
\end{proof}

\begin{fact} \label{fact4}  When $(A, \frak m)$ is an excellent local domain
and $t$ is an element of $\frak m$ such that $A/tA$ satisifes $(R_1)$, then the integral closure of $A/tA$  in its total ring of quotients is local. (See A. Grothendieck \cite[$\S$XIII]{Gr2}, or \cite[Proposition 3.9]{HH}.)   It follows from~\cite[Corollary 37.6]{N} that $t$ is a
prime element. 
\end{fact}

The fact above, that the $S_2$-ification of $A/tA$ is local, does not easily generalize to $A/I$ for a non-principal ideal $I$ of $A$, even under the assumptions that $A/I$ is complete, local, and equidimensional.  One would need to know, for example, that the canonical module of $A/I$ is indecomposable; or more importantly, that any of the equivalent conditions below hold for the ring $A/I$.

\begin{fact} \cite[Theorem 3.6]{HH} \label{HH} For $(T, \frak n)$ a complete local equidimensional ring, the following conditions are equivalent:
\begin{enumerate}[(i)]
\item $H^{\dim T}_{\frak n}(T)$ is indecomposable;
\item The canonical module 
of $T$ is indecomposable;
\item The $S_2$-ification of $T$ is local;
\item For every ideal $J$ of height at least two, $\Spec(T) -V(J)$ is connected;
\item Given any two distinct minimal primes $\frak p, \frak q$ of $T$, there is a sequence of minimal primes $\frak p = \frak p_0, \dots, \frak p_r = \frak q$ such that for $0 \leq i < r$, $\Ht(\frak p_i + \frak p_{i+1}) = 1$.
\end{enumerate}
\end{fact}

We apply this result in Section 4.  (See Lemmas \ref{lem1}, \ref{lem2}, and \ref{prop1}.)

\subsection*{Unramified and \'etale extensions, and isomorphisms in codimension one.}

\

For the remainder of this section, let  $R$ be a ring, and let $S$ be an $R$-algebra.  

\begin{definition}  For a prime ideal $P$ of $S$ and $\frak p = P \cap R$, we say that $P$ is {\it unramified over $R$} if $\frak p S_P = P S_P$ and  $S_P/\frak p S_P$ is a separable field extension of $R_{\frak p}/ \frak p R_{\frak p}$. 
\end{definition}

\begin{definition}  A prime ideal $\p$ of $R$ is {\it unramified in $S$} if (1) every prime lying over $\p$ is unramified over $R$, and (2)
there are only finitely many prime ideals of $S$ lying over $\fp$.
(Note that if $R\into S$ is module finite, then condition (2)  is automatic.)  
\end{definition}

\begin{definition} For $S$ to be {\it unramified over} $R$ means that (1) every prime ideal of $S$ is unramified over $R$, and (2)
for every $\p\in\Spec(R)$ there are only finitely many prime ideals of $S$ lying over $\fp$.
(Note that if $R\into S$ is module finite, then condition (2)  is automatic.)  
\end{definition}

\begin{definition} We say that $S$ is {\it unramified in codimension $i$} (or more accurately, in codimension less than or equal to $i$) over $R$ if every prime ideal of $S$ with height less than or equal to $i$ is unramified over $R$.
\end{definition}

\begin{definition} For $S$ to be {\it \'etale} over $R$ means that $S$ is unramified and flat over $R$.
\end{definition}

\begin{definition}  We say that $S$ is {\it \'etale in codimension one} (or more accurately, in codimension less than or equal to one) over $R$ if, for every prime ideal $P$ of $S$ with height less than or equal to one and $\frak p = P \cap R$, the ring $S_P$ is \'etale over $R_{\frak p}$.
\end{definition}

The next fact is a version of ``purity of branch locus'' from Auslander and Buchsbaum~\cite[Corollary 3.7]{AB} that is very useful for the proof of Theorem~\ref{thmtwo}.

\begin{fact} \label{purity130108a}
Assume that $R$ is local and that $R \hookrightarrow S$ be a module finite extension of normal domains which is unramified in codimension one.  If $S$ is free as a $R$-module, then $S$ is unramified over $R$.
\end{fact}

\begin{definition}  \label{codimdef} Let $M$, $N$ be finitely generated modules over a ring $R$.  The map $\varphi : M \to N$ is said to be an {\it isomorphism in codimension one} if for each prime ideal $\p$ in $R$ of height less than or equal to one, the induced homomorphism $\varphi_{\p}\colon M_{\p}\to N_{\p}$ is an isomorphism.
\end{definition}

The next fact follows with a bit of work from a result of Auslander and Buchsbaum~\cite[Proposition 3.4]{AB}.
It is a useful tool in establishing many of the ring isomorphisms in our main argument.

\begin{fact} \label{AG} Let $A$ be 
a  normal domain.  If $M$ 
is a reflexive $A$-module, $N$ is a torsion-free $A$-module, and 
$\psi\colon M \to N$ is an $A$-module homomorphism, then $\psi$ is an isomorphism if and only if
it is an isomorphism in codimension one.
\end{fact}

\section{Proof of Theorem \ref{thmone}}  \label{sec02}

\begin{chunk} \label{assumptions}
Let $(A, \frak m)$ be a complete, local, normal domain such that $A/\frak m$ is separably closed; let $I$ be a prime ideal of $A$ with finite projective dimension such that $\bar A = A/I$ is normal. Let $e >1$ be an integer which represents a unit in $A$ and assume that $A$ contains a primitive $e$-th root of unity. Let $\frak a$ be a non-zero reflexive ideal of $A$ such that $[\ddual{(\frak a \otimes_A \bar A)}{\bar A}]$ is trivial in $\Cl(\bar A)$, and suppose that the order of $[\frak a]$ in $\Cl(A)$ is $e$.  Thus, $\frak a^{(e)} = aA$ for some $a \in \frak a$.  Set $K$ to be the fraction field of $A$.  
\end{chunk}

Since $A$ is a normal ring, it is the intersection over all height one primes $\p$ of the family of discrete valuation rings $\{A_{\frak p}\}$, each of which is contained in $K$, with corresponding valuations $v_{\frak p}$.  There are height one primes $\frak p_i$ of $A$ and positive integers $n_i$ such that $\frak a = \frak p_1^{(n_1)} \cap \dots \cap \frak p_r^{(n_r)}$.  Then $\frak a^{(e)} = \frak p_1^{(en_1)} \cap \dots \cap \frak p_r^{(en_r)} = aA$ and $v_{\frak p_i}(a) = en_i$.  By the Approximation Theorem \cite[Theorem 12.6]{M}, there is an element $b \in A$ such that $v_{\frak p_i}(b) = n_i$ for $i = 1, \dots, r$ and $v_{\frak q}(b) \geq 0$ for all other height one primes $\frak q$ of $A$.  Set $u = b^e/a\in K$.
Then $v_{\frak p_i} (u) = en_i - en_i =0$  for $i=1, \dots, r$ and $v_{\frak q}(u) = v_{\frak q}(b^e) \geq 0$ for all other $\frak q$.  

Let $\sqrt[e]{u}$ be a fixed $e$-th root of $u$ in an algebraic closure of $K$.
Then $\sqrt[e]{a}:=b/\sqrt[e]{u}$ is an $e$-th root of $a$.
Since $au = b^e$, we have $K[\root e \of a]=K[\root e \of u].$  Moreover, the ring 
\begin{equation} \label{eq121019a}
R = A \oplus \left[\frak a \cdot \frac{\root e \of u}{b}\right] \oplus \left[\frak a^{(2)} \cdot \frac{\root e \of {u^2}}{b^2}\right]\oplus \dots \oplus \left[\frak a^{(e-1)} \cdot \frac{\root e \of {u^{e-1}}}{b^{e-1}}\right]
\end{equation}
is the integral closure of $A$ in the field extension $K[\root e \of a]$, as per \cite[Theorem 2.4]{G3}.
In particular, $R$ is a domain.  
It is worth describing the ring structure on $R$:
for $a_i\in\fa^{(i)}$ and $a_j\in\fa^{(j)}$, we have
$$\left(a_i\frac{\root e \of {u^i}}{b^i}\right)\left(a_j\frac{\root e \of {u^j}}{b^j}\right)
=\begin{cases}
\displaystyle a_ia_j\frac{\root e \of {u^{i+j}}}{b^{i+j}} &\text{if $i+j<e$} \\
\displaystyle a_ia_j\frac{\root e \of {u^{i+j}}}{b^{i+j}}=\frac{a_ia_j}{a}\frac{\root e \of {u^{i+j-e}}}{b^{i+j-e}} &\text{if $i+j\geq e$}.
\end{cases}$$

\begin{lemma}  \label{AtoRetalecodimone}
Let $P \in \Spec(R)$ and set $\p=P\cap A$.  Then $\Ht \fp = \Ht P$.  Furthermore, if $A_{\p}$ is regular, then the extensions $A_{\p}\to R_{\p}$ and $A_{\p}\to R_P$ are \'etale.  In particular, the extension $A \hookrightarrow R$ is \'etale in codimension one.
\end{lemma}

\begin{proof} Since $A$ is normal and $R$ is a domain that is module finite over $A$, the theorems of going-up and going-down apply, and hence $\Ht \fp = \Ht P$.

Now assume that $A_{\fp}$ is regular.  There are two cases to consider:

Suppose $a \notin \frak p$.  Then $a$ is a unit in $A_{\frak p}$.  By \cite[Theorem III.4.4, p.~113]{DI}, the polynomial $T^e - a$ is separable in $A_{\frak p}[T]$.  Thus, $A_{\frak p}[T]/(T^e-a)A_{\frak p}[T]$ is a separable $A_{\frak p}$-algebra, by definition, cf.~\cite[p.~109]{DI}.  Since $a \notin \frak p$, we have $\fa_{\fp}=A_{\fp}$,
so $R_{\fp}\cong R \otimes_A A_{\frak p} \cong A_{\frak p} \oplus A_{\frak p} \oplus \dots \oplus A_{\frak p}$ ($e$ copies); i.e., $R_{\frak p} \cong A_{\frak p}[T]/(T^e-a)A_{\frak p}[T]$.  In other words, $R \otimes_A A_{\frak p}$ is a separable $A_{\frak p}$-algebra.  Set $R_{\frak p} = R \otimes_A A_{\frak p}.$  By definition \cite[p.~40]{DI}, $R_{\frak p}$ is a projective $(R_{\frak p} \otimes_{A_{\frak p}} R_{\frak p})$-module.  In addition, it is the integral closure of $A_{\frak p}$ in $K[\root e \of a] = K[\root e \of u]$.  Thus, by \cite[Proposition 2.2]{AB}, $R_{\frak p}$ is unramified over $A_{\frak p}$.  

On the other hand, if $a \in \frak p$, then $\frak p = \frak p_i$ for some $i$.  Let $S$ denote the complement in $A$ of $\frak p_1, \dots, \frak p_r$, the height one primes in the decomposition of $\frak a$ i.e., $S = A - (\frak p_1\cup \cdots\cup \frak p_r)$.  Note that $u$ is a unit in $A_S$ because $v_{\frak p_i}(u) = 0$ for all $i = 1, \dots, r$.  As above, 
$T^e - u$ is separable in $A_S[T]$.  Thus, $A_S[T]/(T^e - u) A_S[T]$ is a separable $A_S$-algebra, by definition, \cite[p. 109]{DI}.  Since $u \in S$, we have $R \otimes_A A_S \cong \oplus^e A_S$ and the proof proceeds as in the case above.  In conclusion, for each $i$, $A_{\frak p_i} \to R_{\frak p_i}$ is unramified.  

The fact that $A_{\p}$ is regular implies that $R_{\p}$ is free because $\fa A_{\p} \cong A_{\p}$.  Thus, $R_{\p}$ is flat over $A_{\p}$.  In particular, the extension $A_{\p}\to R_{\p}$ is \'etale.  Moreover, it follows that $R_P$ is also flat over $A_{\p}$, and hence the extension $A_{\p}\to R_P$ is \'etale.  Finally, since $A_{\p}$ is regular for any prime $\p = P \cap A$, where $\Ht_R P \leq 1$, it follows that $A \to R$ is \'etale in codimension  one. 
\end{proof}

\begin{lemma} \label{Rlocal}
The ring $R$ is a complete local normal domain.
\end{lemma}

\begin{proof}  
Since $R$ is the integral closure of $A$ in $K[\root e \of a]$,
it is normal.  Moreover, $R$ is complete since it is a module finite extension of the complete ring $A$.  It remains to show that $R$ is local.

Since $R$ is the integral closure of $A$ in $K[\sqrt[e]{a}]$, it follows that
$R$ is the integral closure of $A[\sqrt[e]{a}]$ in $K[\sqrt[e]{a}]$; that is, in the field of
fractions of $A[\sqrt[e]{a}]$.  Note that $A[\sqrt[e]{a}]$ is an integral domain and finitely generated as an $A$-module. Therefore, since $A$ is complete,
it follows that $A[\sqrt[e]{a}]$ is also complete, both in the $\m$-adic topology and in
the $J$-adic topology, where $J$ is the Jacobson radical of $A[\sqrt[e]{a}]$. 
We will show that $A[\sqrt[e]{a}]$ is local and then apply \cite[Corollary 37.6]{N}, which states
that the number of maximal
ideals in $R$, the integral closure of $A[\sqrt[e]{a}]$, is equal to the number of minimal prime ideals in
the completion of $A[\sqrt[e]{a}]$. Since $A[\sqrt[e]{a}]$ is complete, $R$ must be local.

To show that $A[\sqrt[e]{a}]$ is local, it suffices to show that $J$ is maximal.
Since $A[\sqrt[e]{a}]$ is module finite over $A$, each of its maximal ideals
of contracts to $\m$ and $A[\sqrt[e]{a}]$ has only finitely many maximal
ideals. It follows that $J\cap A=\m$. Consider the induced homomorphism
$\psi\colon A/\m \to A[\sqrt[e]{a}]/J$, which we claim is bijective.  Injectivity is satisfied because $J \cap A = \frak m$.
For surjectivity, consider an element $\displaystyle{\xi=\sum_{i=0}^{e-1}a_i (\sqrt[e]{a})^i}$ in $A[\sqrt[e]{a}]$.
Since $(\sqrt[e]{a})^e=a\in\m\subseteq J$, we have $\sqrt[e]{a}\in J$, and hence  $a_i(\sqrt[e]{a})^i\in J$, for $1 \leq i \leq e-1$. 
In $A[\sqrt[e]{a}]/J$, we then have
$$\ol\xi=\sum_{i=0}^{e-1}\ol{a_i}\ol{(\sqrt[e]{a})^i}=\ol{a_0}=\psi(\ol{a_0}).$$
Thus, $\psi$ is surjective, hence an isomorphism, and so $A[\sqrt[e]{a}]$ is local.
\end{proof}

\begin{lemma}  \label{barRequidim}
Set $\bar R=R \otimes_A \bar A=R/IR$.  Let $\bar P\subset \bar R$ be a prime ideal  and set
$\bar\p=\bar P\cap \bar A$. Then $\Ht \bar\p=\Ht \bar P $. Furthermore, the ring $\bar R$ 
is equidimensional, complete, and local. In particular, Fact~\ref{HH} applies with $T=\ol R$.
\end{lemma}

\begin{proof}  
By Lemma~\ref{Rlocal}, the ring $\bar R$ is local and complete.  Since $A$ embeds in $R$ as the degree zero
direct summand, we see that $\bar A$ embeds in $\bar R$ as the degree zero direct summand.
It follows that the induced map $\bar A\to\bar R$ is a module finite monomorphism.  Hence, going-up holds, and as a result $\Ht \bar \p \geq\Ht \bar P$.

We consider the following commutative diagram of local ring homomorphisms,
$$\xymatrix{
A\ar[r]\ar[d]_{\tau} & R \ar[d]^{\pi} \\
\bar A\ar[r] & \bar R}$$
where the vertical maps are the natural surjections,
and the horizontal maps are the inclusions. Set $P=\pi^{-1}(\bar P)$
and $\p=\tau^{-1}(\bar \p)=P\cap A$.

Note that the map $A\to R$ 
satisfies going-down because $A$ is integrally closed and $R$ is an integral domain.
Let $0=\bar{\p}_0\subset\bar{\p}_1\subset\cdots\subset\bar{\p}_h=\bar\p$ be a chain of 
prime ideals of $\bar A$ such that $h=\Ht \bar \p$. It follows that
$I=\p_0\subset\p_1\subset\cdots\subset\p_h=\p$ is a chain of prime ideals of $A$.
By going-down, there are prime ideals
$P_0\subset P_1\subset\cdots\subset P_h=P$ of $R$ such that
$P_i\cap A=\p_i$ for each $i$. In particular, we have $I=\p_0=P_0\cap A$ and so
$IR\subseteq P_0$. It follows that the ideals
$\bar{P_i}=P_i\bar R\subset\bar R$ are prime and form a chain
$\bar{P_0}\subset\bar{P_1}\subset\cdots\subset\bar{P_h}=\bar P$
and thus $\Ht \bar\p=h\leq\Ht \bar P$, as desired.

To show that $\bar R$ is equidimensional, assume that $\bar P$ is a minimal prime
of $\bar R$.  Then $\Ht \bar\p=\Ht \bar P=0$ and so $\bar\p=0$. Thus, the induced
map $\bar A\to \bar R/\bar P$ is a module finite monomorphism.
It follows that $\dim(\bar R/\bar P)=\dim \bar A=\dim \bar R$.
\end{proof}

\begin{lemma}  \label{barAtobarRetalecodimone}
For any prime ideal $\bar P\subset\bar R$ of height less than or equal to one and $\bar \p = \bar P \cap \bar A$, the extensions $\bar{A}_{\bar\p}\to\bar R_{\bar \p}$ and $\bar{A}_{\bar\p}\to\bar R_{\bar P}$ are \'etale.  In particular, the extension $\bar A\to \bar R$ is \'etale in codimension one.  Moreover, the ring $\bar R$ satisfies the Serre condition $(R_1)$.
\end{lemma}

\begin{proof}  
Let $\bar P\subset\bar R$ be a prime of height less than or equal to one.
Let $P$, $\bar \p$ and $\p$ be as in the proof of 
Lemma~\ref{barRequidim}, hence $\Ht \bar\p=\Ht \bar P$. Since $\bar A$ is a normal domain,
it follows that $\bar A_{\bar\p}$ is regular.   The mapping $A_{\frak p} \to \bar A_{\frak p} = \bar A_{\bar{\frak p}}$ has finite projective dimension since, by hypothesis, $A \to \bar A$ has finite projective dimension.  It now follows from \cite[Theorem 6.1(1)]{AFHa} that $A_{\frak p}$ is regular.  Next, by Lemma~\ref{AtoRetalecodimone}, the extensions $A_{\p} \to R_{\p}$ and $A_{\p}\to R_{P}$ are \'etale. It follows by base-change 
that the extensions $\bar{A}_{\bar\p}\to\bar R_{\bar \p}$ and $\bar{A}_{\bar\p}\to\bar R_{\bar P}$ are \'etale.
Thus, the extension $\bar A\to \bar R$ is \'etale in codimension one.
Moreover, the extension $\bar{A}_{\bar\p}\to\bar R_{\bar P}$ has a 
regular closed fibre (in fact, the closed fibre is a field because the extension
is \'etale).  Since the extension is also flat, as it is \'etale, and $\bar A_{\bar\p}$ is regular, it follows that $\bar R_{\bar P}$ is regular.  Therefore, $\bar R$ satisfies $(R_1)$.
\end{proof}

\begin{lemma}  \label{kernelbidualmap}
For the  biduality map $\sigma \colon \bar R \to \ddual{\bar R}{\bar A}$,
we have $\ker \sigma =  \nl(\bar R)  =  j(\bar R)$; see Definition~\ref{j-ideal}.
\end{lemma}

\begin{proof}  
Recall that $j(\bar R) = \{\bar r \in \bar R : \dim(\bar R/ \ann_{\bar R}(\bar r)) < \dim \bar R \}$ is an ideal of $\bar R$. 

First of all, to see that $\nl(\bar R)\subseteq j(\bar R)$, let $\bar r\in \nl(\bar R)$. By the previous lemma, 
the extension $\bar A_{(0)\bar A}\to \bar R_{(0)\bar A}$ is \'etale.  It is also module finite. 
Since $\bar A_{(0)\bar A}$ is a field, the ring
$\bar R_{(0)\bar A}$ is a finite product of fields; hence, it is reduced.  Thus, the image of $\bar r$ in
$\bar R_{(0)\bar A}$ is zero. In other words, 
we have $(\bar r\bar R)_{(0)\bar A}=\bar r\bar R_{(0)\bar A}=(0)$
and so $(0)\bar A\not\in\supp_{\bar A}(\bar r\bar R)$. It follows that
$\dim_{\bar A}(\bar r\bar R)<\dim(\bar A)$.  Now $\dim_{\bar R}(\bar r\bar R) = \dim_{\bar A}(\bar r\bar R)$ and $\dim \bar R =\dim \bar A$, hence $\bar r\in j(\bar R)$.

On the other hand, let $\bar r\in j(\bar R)$.  We need to show that $\bar r\in\bar P$ for each
prime ideal $\bar P\subset\bar R$.  If it is not, then $\bar r\not\in\bar{P_0}$, for some minimal prime $\bar {P_0}$ of $\bar R$.  The element $\bar r$ represents a unit in $\bar R_{\bar{P_0}}$. In particular, we have
$0\neq\bar r/1\in\bar R_{\bar{P_0}}$ and so $\bar{P_0}\in\supp_{\bar R}(\bar r\bar R)$. 
It follows that $\dim_{\bar R}(\bar r \bar R) = \dim \bar R$ since $\bar R$ is equidimensional.
This contradicts the fact that $\bar r\in j(\bar R)$.
Therefore, $j(\bar R) \subseteq \nl(\bar R)$.

We now show that $\ker \sigma=j(\bar R)$. 
The module $\bar R_{(0)\bar A}$ is free over $\bar A_{(0)\bar A}$, hence the localized map
$\sigma_{(0)\bar A}\colon \bar R_{(0)\bar A}\to(\ddual{\bar R}{\bar A})_{(0)\bar A}$
is an isomorphism.  Thus, $(\ker \sigma)_{(0)\bar A} \cong \ker(\sigma_{(0)\bar A}) = (0)$, so $\dim_{\bar A}(\ker \sigma)<\dim \bar A$.  Now $\dim_{\bar R}(\ker \sigma) = \dim_{\bar A}(\ker \sigma)$ and $\dim \bar R =\dim \bar A$, hence $\ker \sigma \subseteq j(\bar R)$. 

For the reverse containment, let $\bar r\in j(\bar R)$. Then 
$\dim_{\bar A}(\bar r\bar R)<\dim \bar A$ as above, which means that $\bar r\bar R_{(0)\bar A}=(0)$.
It follows that there exists a non-zero element $\bar b\in\bar A$ such that $\bar b\bar r=0$, and so
$\dim_{\bar A}(\bar r\bar A)<\dim \bar A$. 
For each $\psi\in\Hom_{\bar A}(\bar R,\bar A)$, 
$\dim_{\bar A}(\psi(\bar r\bar A))\leq\dim_{\bar A}(\bar r\bar A)<\dim \bar A$, since  $\bar r\bar A \twoheadrightarrow \psi(\bar r\bar A)$ via $\psi$. 
However, $\psi(\bar r\bar A)$ is an $\bar A$-submodule of $\bar A$; that is,
an ideal of $\bar A$. If $\psi(\bar r\bar A) \neq (0)$, then $\psi(\bar r\bar A) \hookrightarrow \psi(\bar r\bar A)_{(0)\bar A}$ because $\bar A$ is a domain; but this would mean that $\psi(\bar r\bar A)_{(0)\bar A} \neq 0$, hence $\dim_{\bar A}(\psi(\bar r\bar A)) = \dim \bar A$.   
This is a contradiction.  Therefore, $\psi(\bar r\bar A)=(0)$.
This says that $\psi(\bar r)=0$ for all $\psi\in\Hom_{\bar A}(\bar R,\bar A)$.
By definition, $\sigma(\bar r)\colon \Hom_{\bar A}(\bar R,\bar A)\to\bar A$ is the 
map $\sigma(\bar r)(\psi)=\psi(\bar r)=0$.  Thus, $\sigma(\bar r)=0$, which means that $\bar r\in\ker \sigma$, as desired.
\end{proof}

\begin{lemma}  \label{barRmodjR1} The ring $\bar R/j(\bar R)$ satisfies $(R_1)$, so its integral closure $(\bar R/j(\bar R))'$ 
in its total ring of quotients is its $S_2$-ification. 
\end{lemma}

\begin{proof}  Take $\tilde{P} = \bar P/j(\bar R)$ a height one prime of $\bar R/j(\bar R)$.  Then $\bar P$ is a height one prime of $\bar R$
since $j(\bar R)=\nl(\bar R)$, and $\big(\bar R/j(\bar R)\big)_{\tilde{P}} = \bar R_{\bar P}/(j(\bar R))_{\bar P}.$  Since $\bar R$ is complete and local, it has a canonical module $\omega_{\bar R}$.  Recalling Fact \ref{fact1}, $j(\bar R) = \ker(\bar R \to \Hom_{\bar R}(\omega_{\bar R}, \omega_{\bar R}))$, hence $j(\bar R)_{\bar Q} \cong j(\bar R_{\bar Q})$ for any prime ideal $\bar Q$ of $\bar R$.  
Note that this uses the fact that $\bar R$ is equidimensional; see Lemma~\ref{barRequidim} and~\cite[(2.2i)]{HH}.
In particular, since $\bar R$ satisfies the $(R_1)$ condition as per Lemma~\ref{barAtobarRetalecodimone},
we have $j(\bar R_{\bar P}) = (0)$ and hence $\big(\bar R/j(\bar R)\big)_{\tilde{P}} \cong \bar R_{\bar P}$ is regular.

Thus, $\bar R/j(\bar R)$ satisfies $(R_1)$.
Lemma~\ref{domainfact} implies that the integral closure of $\bar R/j(\bar R)$  in its total ring of quotients is the $S_2$-ification of $\bar R/j(\bar R)$, that is, of $\bar R$. 
\end{proof}

\begin{lemma}  \label{barRfree} The $\bar A$-module $\ddual{\bar R}{\bar A}$ is free of rank $e$ and thus satisfies $(S_2)$ over $\bar A$.
\end{lemma}

\begin{proof}  Since $[\frak a] \in \ker(\Cl(A) \stackrel{\gamma}{\to} \Cl(\bar A))$, we have $\ddual{(\frak a \otimes_A \bar A)}{\bar A} \cong (\bar{\alpha})$, for some  $\bar{\alpha} \in \bar A$.  Thus, using the fact that $\gamma$ is a group homomorphism, $\ddual{(\frak a^{(t)} \otimes_A \bar A)}{\bar A} = \gamma([\frak a^{(t)}]) = \big(\gamma([\frak a]) \big)^t \cong (\bar{\alpha}^t)$, for all $t$; i.e.,
each component of $\ddual{\bar R}{\bar A}$ is a free $\bar A$-module.  Consequently, $\ddual{\bar R}{\bar A}$ is free of rank $e$ and satisfies the $(S_2)$ condition, since $\bar A$ is $(S_2)$.
\end{proof}

\begin{lemma} \label{lem130109a} 
For each prime $\tilde P$ in $(\bar R/j(\bar R))'$, we have $$\Ht(\tilde P\cap\bar A)=\Ht(\tilde P\cap(\bar R/j(\bar R)))=\Ht \tilde P.$$
Moreover, $(\bar R/j(\bar R))'$ satisfies $(S_2)$ as an $\bar A$-module and as a ring.
\end{lemma}

\begin{proof}  
The equality $\Ht(\tilde P\cap(\bar R/j(\bar R)))=\Ht \tilde P$ is from~\cite[Proposition 3.5(a)]{HH}.
Lemma~\ref{barRequidim} implies that $\Ht(\tilde P\cap\bar A)=\Ht \tilde P$. Since $(\bar R/j(\bar R))'$ satisfies $(S_2)$ as an $\bar R/j(\bar R)$-module by Lemma~\ref{barRmodjR1},
it follows from~\cite[Corollaire (5.7.11)]{Gr4-2} that $(\bar R/j(\bar R))'$ satisfies $(S_2)$ as an $\bar A$-module and as a ring.
\end{proof}

\begin{lemma} \label{primedualiso} There are $\bar A$-isomorphisms $(\bar R/j(\bar R))' \cong \ddual{(\bar R/j(\bar R))}{\bar A}\cong\ddual{\bar R}{\bar A}$.
\end{lemma}

\begin{proof}  Let $\bar \p$ be a prime of height less than or equal to one in $\bar A$. Recall that $\bar A_{\bar \p} \to \bar R_{\bar \p}$ is \'etale, by Lemma \ref{barAtobarRetalecodimone}, therefore, $\bar R_{\bar \p}$ is free over $\bar A_{\bar \p}$.  The maximal ideals of $\bar R_{\bar \p}$ are of the form $\bar P \bar R_{\bar \p}$, where $\bar P \cap \bar A = \bar \p$, thus $\Ht_{\bar R} \bar P = 1$ by Lemma~\ref{barRequidim}.  Since $(\bar R_{\bar \p})_{\bar P \bar R_{\bar \p}} \cong \bar R_{\bar P}$ and $\bar R$ satisfies the ($R_1$) condition, $\bar R_{\bar P}$ is a regular local ring.  Therefore, $\bar R_{\bar \p}$ is a regular ring.   Moreover, $j(\bar R_{\bar P}) = 0$ since $\bar R_{\bar P}$ is a domain; hence, $j(\bar R_{\bar \p})_{\bar P \bar R_{\bar \p}} = 0$ for all 
$\bar P \bar R_{\bar \p}$ (as in the proof of Lemma \ref{barRmodjR1}).  We conclude that $j(\bar R)_{\bar \p} = 0$.  
Consequently, $(\bar R/j(\bar R))_{\bar \p} = \bar R_{\bar \p}$.  Since $\bar R_{\bar \p}$ is regular, hence normal, we deduce that $((\bar R/j(\bar R))')_{\bar \p} = 
(\bar R_{\bar \p}/j(\bar R)_{\bar \p})' = \bar R_{\bar \p}$.

It follows that the $\bar A$-homomorphisms $\ddual{\bar R}{\bar A}\to\ddual{(\bar R/j(\bar R))}{\bar A} \to \ddual{((\bar R/j(\bar R))')}{\bar A}$ are isomorphisms in codimension one.  Note that $\bar R/j(\bar R)$ is complete, local, and reduced, since it was shown in Lemma \ref{kernelbidualmap} that $j(\bar R) = \nl(\bar R)$.  Therefore, $(\bar R/j(\bar R))'$ is finite over $\bar R/j(\bar R)$, (see, e.g., \cite[p.~263]{M}), and hence finitely generated over $\bar A$.  It follows that the maps $\ddual{\bar R}{\bar A}\to\ddual{(\bar R/j(\bar R))}{\bar A} \to \ddual{((\bar R/j(\bar R))')}{\bar A}$ are homomorphisms of reflexive $\bar A$-modules and  isomorphisms in codimension one.  Since reflexive implies torsion free (see e.g., \cite[Corollary 3.7]{EG}), these maps are isomorphisms, as per Fact~\ref{AG}.

Lemma~\ref{lem130109a} says that $(\bar R/j(\bar R))'$ satisfies the $(S_2)$ condition as an $\bar A$-module.  Hence it is reflexive as an $\bar A$-module, by \cite[Theorem 3.6]{EG}.  Thus, we have $(\bar R/j(\bar R))' \cong \ddual{((\bar R/j(\bar R))')}{\bar A} \cong \ddual{(\bar R/j(\bar R))}{\bar A}$ as $\bar A$-modules. 
\end{proof}

\begin{remark} The upshot of the previous lemma is that $(\bar R/j(\bar R))'$ is the $S_2$-ification of $\bar R$.  (See Definition~\ref{S2def}.)  Therefore, $(\bar R/j(\bar R))'$ is local, by assumption.
\end{remark}

\begin{lemma} \label{barAtodualbarRetalecodimone} The map $\bar A \to (\bar R/j(\bar R))'$ is \'etale in codimension one. 
\end{lemma}

\begin{proof}  Recall the composition $\bar A \to \bar R \to \bar R/j(\bar R) \to (\bar R/j(\bar R))'$.  Let $\bar \p$ be a prime ideal of $\bar A$ with height less than or equal to one.  By Lemma \ref{barAtobarRetalecodimone}, $\bar A_{\bar \p} \to \bar R_{\bar \p}$ is \'etale.  
Next, $((\bar R/j(\bar R))')_{\bar \p} = (\bar R_{\bar \p}/j(\bar R_{\bar \p}))' = \bar R_{\bar \p} \cong (\bar R/j(\bar R))_{\bar \p}$ since $j(\bar R)_{\bar \p}=j(\bar R_{\bar \p}) = 0$; see the proof of Lemma \ref{primedualiso}.  Hence the composition $\bar A_{\bar \p} \to ((\bar R/j(\bar R))')_{\bar \p}$ is \'etale.

Now, let $\tilde P$ be a prime ideal of $(\bar R/j(\bar R))'$ with height less than or equal to one, and set $\bar\fp=\tilde P\cap\bar A$.  Lemma~\ref{lem130109a}
implies that $\Ht \bar\fp \leq 1$, so the map $\bar A_{\bar \p} \to ((\bar R/j(\bar R))')_{\tilde P}$ is \'etale, by the previous paragraph.
\end{proof}

\begin{lemma}  \label{normalsmap}
The ring $(\bar R/j(\bar R))'$ is a normal domain.
\end{lemma}

\begin{proof}  
Lemma~\ref{lem130109a} implies that $(\bar R/j(\bar R))'$ is $(S_2)$ as a ring. We claim that it is also $(R_1)$ as a ring.
To this end, let $\tilde P$ be a height one prime of $(\bar R/j(\bar R))'$, and set $\bar\fp=\tilde P\cap\bar A$. 
Then we have $\Ht \ol\fp=1$, again by Lemma~\ref{lem130109a}, so $\bar A_{\bar\fp}$ is regular.
From
Lemma~\ref{barAtodualbarRetalecodimone}, we know that $((\bar R/j(\bar R))')_{\tilde P}$ is \'etale over $\bar A_{\bar\fp}$.
Since $\bar A_{\bar\fp}$ is regular, this implies that $((\bar R/j(\bar R))')_{\tilde P}$ is regular.

Using Serre's criterion for normality, it follows that $(\bar R/j(\bar R))'$ is a finite direct product of normal domains; see~\cite[p.~64]{M}.
However, since $(\bar R/j(\bar R))'$ is local, by assumption, it can not decompose into a non-trivial product of domains.  Hence it is a normal domain.
\end{proof}

\begin{lemma}  \label{barAtobarRetale}
The extension $\bar A \hookrightarrow (\bar R/j(\bar R))'$ is \'etale. 
\end{lemma}

\begin{proof}
Lemma~\ref{normalsmap} implies that $(\bar R/j(\bar R))'$ is a normal domain.
The extension $\bar A \to (\bar R/j(\bar R))'$ between normal domains is module finite, free, and \'etale in codimension one by Lemmas~\ref{barRfree}, \ref{primedualiso},
and~\ref{barAtodualbarRetalecodimone}. Thus, it is \'etale by purity of branch locus, specifically, Fact~\ref{purity130108a} applies since $(\bar R/j(\bar R))'$ is a free $\bar A$ module, and hence projective over $\bar A$.
\end{proof}

{\it Conclusion of proof of Theorem \ref{thmone}.}
Let $S$ denote $(\bar R/j(\bar R))'$.  Since the extension $\bar A \hookrightarrow S$ is \'etale and local, the induced field extension $\bar A/\bar \fm \to S/\bar \fm S$ is separable and algebraic.  Since $\bar A/\bar \fm$ is separably closed, this induced map is an isomorphism.  Recall that $S$ is free over $\bar A$ of rank $e$.  Consequently, there are isomorphisms $$\bar A/\bar \fm \cong S/\bar \fm S \cong (\bar A/\bar \fm)^e.$$
It follows that $e=1$, but this contradicts the assumptions in paragraph \ref{assumptions}. \qed

\section{Proof of Theorem \ref{thmtwo}}
\label{sec130110a}

\begin{chunk} \label{newassumptions}
Let $(A, \frak m)$ be an excellent, local, normal domain and $I$ a prime ideal of $A$ with finite projective dimension such that:
\begin{enumerate}
\item[(i)] $\bar A = A/I$ is normal; 
\item[(ii)] $I$ is a complete intersection on the punctured spectrum of $A$; and
\item[(iii)] $\mu(I) \leq \dim A-2$.
\end{enumerate}
Let $e >1$ be an integer which represents a unit in $A$.
Let $\frak a$ be a non-zero reflexive ideal of $A$ such that $[\ddual{(\frak a \otimes_A \bar A)}{\bar A}]$ is trivial in $\Cl(\bar A)$, and suppose that the order of $[\frak a]$ in $\Cl(A)$ is $e$.  Thus, $\frak a^{(e)} = aA$ for some $a \in \frak a$.  Set $K$ to be the fraction field of $A$.  
\end{chunk}

The theorem is trivial if $I= (0)$, therefore assume $I$ is non-zero.  If $\dim A \leq 2$, then the $(R_1)$ condition on $\bar A$ forces the factor ring to be regular, and hence $A$ must itself be regular by~\cite[Theorem 6.1(1)]{AFHa}.  Consequently, $\Cl(A) = 0$, which again provides a trivial result.  Therefore, 
assume that $\dim A \geq 3$.  We make a series of reductions to reduce to the case of Theorem \ref{thmone}, and as in the previous proof, ultimately obtain our conclusion by way of contradiction.

The prime ideal $I$ has finite projective dimension, therefore
$$\pd_{A_I}(A_I/IA_I)\leq\pd_A \bar A<\infty.$$
It follows that the localization $A_I$
is regular and hence a unique factorization domain.  This implies
that $\fa A_I=bA_I$ for some $b$ in $\fa$, and hence
$$b^eA_I=(bA_I)^{(e)}=(\fa A_I)^{(e)}=\fa^{(e)}A_I=aA_I.$$

\begin{lemma}  \label{Acomplete}
$A$ can be assumed to be complete.
\end{lemma}

\begin{proof}  
Because $A$ is an excellent local normal domain, the completion $\hat{A}$ is a complete local normal domain by \cite[Corollary 37.6]{N}. 
Next, $\pd_{\hat{A}} I \hat{A} < \infty$ since that map $A \to \hat{A}$ is flat.  Moreover, because the map is also local, $\mu_{\hat{A}}(I \hat{A}) = \mu_A(I) \leq \dim A -2 = \dim \hat{A} - 2$. We need to show that $I \hat{A}$ is a complete intersection on the punctured spectrum of $\hat{A}$, denoted $\Spec^{\circ}(\hat{A})$.

Let $P \in \Spec^{\circ}(\hat{A})$ and set $\frak p = P \cap A$.  Since the closed fibre $\hat{A}/\frak m \hat{A}$ is isomorphic to $A/\frak m$, the fact that $P$ is not maximal implies that $\frak p \neq \frak m$.  Thus, the ideal $I A_{\frak p}$ is either $A_{\fp}$ or generated by an $A_{\fp}$-regular sequence.  If $I A_{\fp}=A_{\fp}$, then $I \hat{A}_P = I A_{\fp} \cdot \hat{A}_P = \hat{A}_P$.  If $I A_{\frak p}$ is generated by an $A_{\fp}$-regular sequence, then $I \hat{A}_P = I A_{\fp} \cdot \hat{A}_P$ is generated over $\hat{A}_P$ by the same sequence, which is $\hat{A}_P$-regular since the induced map $A_{\frak p} \to \hat{A}_P$ is flat and local.

The left-most diagram below is commutative and each map has finite flat dimension:
\begin{equation} 
\begin{split}
\label{cd1}
\xymatrix{
A \ar[r] \ar[d]_{\rho} & \bar A \ar[d] & & \Cl(A) \ar[r] \ar@{^{(}->}[d]_{\Cl(\rho)} & \Cl(\bar A) \ar@{^{(}->}[d]  \\
\hat{A} \ar[r] & \abarhat & & \Cl(\hat{A}) \ar[r] & \Cl(\widehat{\bar A}). }
\end{split}
\end{equation}

It follows that the composition $A \to \abarhat$ also has finite flat dimension, so \cite[Theorem 1.14]{SWS} implies that the second diagram above also commutes.

Suppose the result holds on $\Cl(\hat{A}) \to \Cl(\widehat{\bar A})$.  Since $[\frak a] \in \ker(\Cl(A) \to \Cl(\bar A))$, it follows that 
$\Cl(\rho)([\frak a]) \in \ker(\Cl(\hat{A}) \to \Cl(\widehat{\bar A}))$.  By properties of group monomorphisms, the order of $\Cl(\rho)([\frak a])$ is $e$,
which is a unit in $\Cl(\hat{A})$.  The condition $e>1$ contradicts the complete case.  
\end{proof}

\begin{lemma}  \label{rootsofunity}
$A$ can be assumed to contain a primitive $e$-th root of unity.
\end{lemma}

\begin{proof}  Suppose that $A$ does not contain a primitive $e$-th root of unity and let $\zeta$ be a primitive $e$-th root of unity in the algebraic closure of $K$, the fraction field of $A$.  
Note that $\zeta$ exists because $e$ is a unit in $A$.
By assumption, $A$ is a complete local normal domain.

As per \cite[Proposition VI.1]{Ra}, the extension $A \to A[T]/(T^e -1)$ is \'etale.  Let $\Phi_e(T) \in K[T]$ be the minimal polynomial of $\zeta$ over $K$.  Since $A$ is integrally closed
and the coefficients of $\Phi_e(T)$ are integral over $A$, it follows that $\Phi_e(T) \in A[T]$.  Using the Division Algorithm in $A[T]$, we have $A[\zeta] \cong A[T]/\Phi_e(T)$, and $\Phi_e(T)$ divides $(T^e - 1)$ in $A[T]$.  Set d = deg $\Phi_e(T)$.  Note that  $\displaystyle{A[T]/\Phi_e(T) \cong \oplus_{i=0}^{d-1} A}$ as $A$-modules, hence the extension $A \to A[T]/\Phi_e(T)$ is faithfully flat.  The maps $A \to A[T]/(T^e -1)$ and $A[T]/(T^e -1) \to A[T]/\Phi_e(T)$ are both unramified, thus the composition  is unramified; i.e., the extension $A \to A[T]/\Phi_e(T)$ is \'etale.

The ring $A[\zeta]$ is excellent since excellence is preserved under finitely-generated extensions.  Moreover, $A[\zeta]$ is a domain since it is a subring of $K[\zeta]$.
Since $A$ is a normal domain and the finite extension $A \to A[\zeta]$ is \'etale, $(A[\zeta])_Q$ is a normal domain for each prime $Q$ of $A[\zeta]$ by \cite[Theorem 23.9]{M}.

\newcommand{\squab}{\widehat{A[\zeta]_Q}}
\newcommand{\quab}{A[\zeta]_Q}
\newcommand{\uab}{A[\zeta]}

Let $Q$ be a maximal ideal of $A[\zeta]$.  Since $A \to A[\zeta]$ is finite, hence $Q \cap A$ is maximal; that is, $Q \cap A = \frak m$.  
We next note that the ring $\squab$ and the ideal $I \squab$ satisfy the hypothesis of the theorem.  Indeed, since $\quab$ is an excellent local normal domain, it follows that $\squab$ is a complete normal local domain.

The maps in the sequence $A \to \uab \to \quab \to \squab$ are all flat, hence the composition $A \to \squab$ is flat.  It follows that the extension $I \squab$ has finite projective dimension over $\squab$.  It remains to show that $\squab/I \squab$ is a normal domain.  To this end, note that $I \quab$ has finite projective dimension over $\quab$.  The map $A \to \quab$ is flat and unramified because $A \to \uab$ is flat and unramified.  Therefore, $A/I \to \quab/I \quab$ is flat and unramified.  The fact that $A/I$ is normal implies that $\quab/I \quab$ is normal; since $\quab/I \quab$ is local, it is also an excellent local normal domain.  Hence its completion $\squab/I \squab$ a local normal domain.

As in Lemma \ref{Acomplete}, we have $\mu_{\squab}(I \squab) \leq \dim \squab -2$ and the fact that the map $A \to \squab$ is flat and local, with closed fibre a field, implies that $I \squab$ is a complete intersection on $\Spec^\circ\left(\squab \right)$.
Likewise, as in Lemma \ref{Acomplete}, we have a commutative diagram
\begin{equation} \label{cd2}
\begin{split}
\xymatrix{
\Cl(A) \ar[r] \ar@{^{(}->}[d] & \Cl(\bar A) \ar@{^{(}->}[d]  \\
\Cl \left(\squab \right) \ar[r] & \Cl \left(\squab/I \squab \right) }
\end{split}
\end{equation}
and the same argument shows that we may replace $A$ with $\squab$ to assume that $A$ contains a primitive $e$-th root of unity.
\end{proof}

\begin{lemma}  \label{ksepclosed}
$A/\m$ can be assumed to be separably closed.
\end{lemma}

\newcommand{\ksep}{\sf k^{\text{sep}}}

\begin{proof}  
Set $\sf k=A/\m$ and let $\ksep$ be a separable closure of $\sf k$; (i.e., the set of all separable elements in a fixed algebraic closure).
This means that $\sf k\subseteq\ksep$ is a separable algebraic extension,
and $\ksep$ has no nontrivial separable algebraic extensions.
Grothendieck~\cite[Proposition (0.10.3.1)]{Gr3-1} shows that there is a flat local ring homomorphism $\rho\colon (A, \m)\to(B,\n)$ such that $\n=\m B$, the ring $B$ is complete,
and $B/\n\cong \ksep$. 

It follows that the extension $\rho$ is regular,
in the terminology of Matsumura~\cite[pp.\ 255--256]{M}. 
To see this, observe that the extension $\sf k\to\ksep$ is 0-smooth by~\cite[Theorem 26.9]{M}.
Using~\cite[Theorem 28.10]{M} it follows that $B$ is $\m B$-smooth over $A$,
that is, that $B$ is $\n$-smooth over $A$. Now, apply the theorem of M.~Andr\'e
from~\cite[Lemme II.57]{A} to conclude that the extension $\rho$ is regular.  (See also \cite[p.~260]{M}.)

From~\cite[Theorem 32.2]{M} it follows that $B$ is a normal domain. 
The same reasoning implies that the ring $B/IB$ is a normal domain.
The fact that the extension $\rho$ is flat and local 
implies that $\pd_{B}(B/IB)=\pd_A(\bar A)<\infty$.

Again as in Lemma \ref{Acomplete}, we have $\mu_{B}(I B) \leq \dim B -2$, and the fact that the map $\rho$ is flat and local, with closed fibre a field, implies that $I B$ is a complete intersection on $\Spec^\circ(B)$.

Now argue as in the proof of Lemma~\ref{Acomplete}: there is a commutative diagram of divisor class groups as in \eqref{cd1} and \eqref{cd2}, but with bottom row $\Cl(B) \to \Cl(B/IB)$.  As $[\fa]\in\ker(\Cl(A) \to \Cl(\bar A))$ has order $e$ in $\Cl(A)$, the element $\Cl(\rho)([\fa])\in\ker(\Cl(B) \to \Cl(B/IB))$ has order $e$ as well.  If the theorem holds for $B\to B/IB$, then we have a contradiction.
\end{proof}

The three lemmas below give the explicit details, implicit in~\cite[proof of Proposition 3.9(c)]{HH}, as to the fact that if the $S_2$-ification of $\bar R$ is not local, then there exists a prime ideal $\tilde{P}$ of $\bar R$ of height at least two such that the punctured spectrum of $\bar R_{\tilde{P}}$ is disconnected.  (It should be noted that all assumptions in these lemmas are as stated and do not depend upon previously set notation in (\ref{newassumptions}).)

\begin{lemma}\label{lem1}
Let $B$ be a complete local equidimensional ring. Assume that the $S_2$-ification of $B$ is not local.
Then there exist ideals $K_1,K_2\subseteq B$ such that 
\begin{enumerate}[\rm(1)]
\item \label{item1}
$\Ht(K_1+K_2)\geq 2$, 
\item \label{item2}
$K_1\cap K_2\subseteq\nl(B)$, and 
\item \label{item3}
$K_j\not\subseteq\nl(B)$ for $j=1,2$.
\end{enumerate}
\end{lemma}

\begin{proof}
Since the $S_2$-ification of $B$ is not local, we know from [14, Theorem 3.6] that there is an ideal $J\subseteq B$ such that
$\Ht J \geq 2$ and $\Spec(B)- V(J)$ is disconnected. Note that we are using the subspace topology on 
$\Spec(B)- V(J)$ induced from the Zariski topology on $\Spec(B)$. It follows that there are non-empty disjoint open 
subsets $U_1,U_2\subsetneq \Spec(B)- V(J)$ such that $\Spec(B)- V(J)=U_1\cup U_2$.
Since each $U_j$ is open in $\Spec(B)- V(J)$, it is of the form $U_j=[\Spec(B)- V(J)]\cap W_j$ for some open subset
$W_j\subseteq\Spec(B)$.
Thus, there are ideals $I_j\subseteq B$ such that $W_j=\Spec(B)- V(I_j)$, and it follows that
\begin{align*}
U_j
&=[\Spec(B)- V(J)]\cap[\Spec(B)- V(I_j)]\\
&=\Spec(B)- [V(J)\cup V(I_j)]\\
&=\Spec(B)- V(JI_j).
\end{align*}
For $j=1,2$, set $K_j=JI_j$.
This implies that $U_j=\Spec(B)- V(K_j)$ for $j=1,2$.

The condition $\Spec(B)- V(J)=U_1\cup U_2$  implies that
\begin{align*}
\Spec(B)- V(J)
&=[\Spec(B)- V(K_1)]\cup[\Spec(B)- V(K_2)]\\
&=\Spec(B)- [V(K_1)\cap V(K_2)]\\
&=\Spec(B)- V(K_1+K_2).
\end{align*}
So we have
\begin{align*}
\Ht(K_1+K_2)
&=\min\{\Ht Q \mid Q\in V(K_1+K_2)\}\\
&=\min\{\Ht Q \mid Q\in V(J)\}\\
&=\Ht J \\
&\geq 2.
\end{align*}
This explains condition~\eqref{item1} from the statement of the lemma.

Next, we use the fact that $U_1$ and $U_2$ are disjoint:
\begin{align*}
\emptyset
&=U_1\cap U_2\\
&=[\Spec(B)- V(K_1)]\cap[\Spec(B)- V(K_2)]\\
&=\Spec(B)- [V(K_1)\cup V(K_2)]\\
&=\Spec(B)- V(K_1\cap K_2)
\end{align*}
so that $V(K_1\cap K_2)=\Spec(B)$.
It follows that, for each $Q\in\Spec(B)$ we have $K_1\cap K_2\subseteq Q$.
It follows that $K_1\cap K_2\subseteq\nl(B)$, i.e., we have condition~\eqref{item2} from the statement of the lemma.

For condition~\eqref{item3}, we argue by contradiction. Suppose that $K_j\subseteq \nl(B)$.
This implies that $V(K_j)=\Spec(B)$, so
$U_j=\Spec(B)- V(K_j)=\emptyset$.
This contradicts our choice of $U_j$. Thus, condition~\eqref{item3} is satisfied.
\end{proof}

\begin{lemma}\label{lem2}
Let $B$ be a complete local equidimensional ring. Assume that the $S_2$-ification of $B$ is not local.
Then there exist ideals $L_1,L_2\subseteq B$ such that 
\begin{enumerate}[\rm(i)]
\item \label{item4}
$\Ht(L_1+L_2)\geq 2$, 
\item \label{item5}
$L_1\cap L_2=\nl(B)$,  
\item \label{item6}
$L_j\not\subseteq\nl(B)$ for $j=1,2$, and
\item \label{item7}
$L_j=P_{j,1}\cap\cdots\cap P_{j,t_j}$  for $j=1,2$, such that $t_j\geq 1$ and each $P_{j,k}\in\min(B)$.
\end{enumerate}
\end{lemma}

\begin{proof}
Let $K_1,K_2\subseteq B$ be as in Lemma~\ref{lem1}.
It follows that
\begin{enumerate}[\rm(a)]
\item \label{item8}
$\Ht(\rad{K_1}+\rad{K_2})\geq\Ht(K_1+K_2)\geq 2$, 
\item \label{item9}
$\rad{K_1}\cap \rad{K_2}=\rad({K_1\cap K_2})\subseteq\rad{\nl(B)}=\nl(B)$,  
\item \label{item10}
$\rad{K_j}\not\subseteq\nl(B)$ for $j=1,2$, and
\item \label{item11}
each $\rad{K_j}$ is an intersection of  primes of $B$.
\end{enumerate}
Thus, we may replace $K_j$ with $\rad{K_j}$ to assume that each $K_j$ 
is an intersection of  primes of $B$.

Claim: each $K_j$ is contained in a minimal prime of $B$. 
We prove this by contradiction. 
By symmetry, suppose that $K_1$ is not contained in any minimal prime of $B$.
Then for each $Q\in\min(B)$, we have
$$K_1K_2\subseteq K_1\cap K_2\subseteq\nl(B)\subseteq Q.$$
Since $Q$ is prime, we have $K_j\subseteq Q$ for some $j=1,2$.
But $K_1\not\subseteq Q$ by assumption, so we must have $K_2\subseteq Q$.
Since $Q$ was chosen arbitrarily from $\min(B)$, we conclude that
$K_2$ is contained in the intersection of the minimal primes of $B$,
that is, $K_2\subseteq\nl(B)$.
This contradicts condition~\eqref{item3} from Lemma~\ref{lem1}.
Thus, the claim is established.

By condition~\eqref{item11} above, for $j=1,2$ we can write $K_j=P_{j,1}\cap\cdots\cap P_{j,n_j}$
where each $P_{j,k}$ is prime.
Re-order the $P_{j,k}$ if necessary to assume that $P_{j,1},\ldots,P_{j,t_j}\in\min(B)$
and $P_{j,t_j+1},\ldots,P_{j,n_j}\notin\min(B)$.
Note that the claim above implies that $t_j\geq 1$.
For $j=1,2$ set $L_j=P_{j,1}\cap\cdots\cap P_{j,t_j}\supseteq K_j$.
By definition of $L_j$, we have condition~\eqref{item7} from the statement of the lemma.
Furthermore, the condition $L_j\supseteq K_j$ implies that
$$\Ht(L_1+L_2)\geq\Ht(K_1+K_2)\geq 2$$
and $L_j\not\subseteq\nl(B)$,
so conditions~\eqref{item4} and~\eqref{item6} are satisfied.

We conclude the proof by verifying condition~\eqref{item5}. Since each $L_j$ is an intersection of  primes of $B$,
we have $L_j\supseteq \nl(B)$, and hence $L_1\cap L_2\supseteq\nl(B)$.
To show the reverse containment, let $Q\in\min(B)$; it suffices to show that
$L_1\cap L_2\subseteq Q$. We know that
\begin{align*}
P_{1,1}\cap\cdots\cap P_{1,n_1}\cap P_{2,1}\cap\cdots\cap P_{2,n_2}
&=K_1\cap K_2
\subseteq\nl(B)\subseteq Q.
\end{align*}
Thus, we have $P_{j,k}\subseteq Q$ for some $j,k$. Since $Q$ is minimal, we must have $P_{j,k}=Q$ so $k\leq t_j$.
From this, we have $L_j=P_{j,1}\cap\cdots\cap P_{j,t_j}\subseteq P_{j,k}=Q$, as desired.
\end{proof}

\begin{lemma}\label{prop1}
Let $B$ be a complete local equidimensional ring. Assume that the $S_2$-ification of $B$ is not local.
Then there is a prime ideal $P$ of $B$ such that $\Ht P \geq 2$ and $\Spec^\circ(B_P)$ is disconnected.
\end{lemma}

\begin{proof}
Let $L_1,L_2$ be as in Lemma~\ref{lem2}, and let $P$ be  minimal in $V(L_1+L_2)$.
It follows that $\Ht P \geq\Ht(L_1+L_2)\geq 2$.
We show that $\Spec^\circ(B_P)$ is disconnected.
For $j=1,2$ we set 
$$V^\circ(L_jB_P)=V(L_jB_P)\cap\Spec^\circ(B_P)=V(L_jB_P)-\{PB_P\}.$$
Since we are using the subspace topology on 
$\Spec^\circ(B_P)$ induced from the Zariski topology on $\Spec(B_P)$,
the sets $V^\circ(L_jB_P)$ are closed in $\Spec^\circ(B_P)$.
We show that the closed sets $V^\circ(L_jB_P)$ give a disconnection of $\Spec^\circ(B_P)$.

Claim 1: $\Spec^\circ(B_P)=V^\circ(L_1B_P)\cup V^\circ(L_2B_P)$.
The containment $\supseteq$ is by definition of $V^\circ(L_jB_P)$.
For the reverse containment, let $QB_P\in\Spec^\circ(B_P)$.
Thus, we have $Q\in\Spec(B)$ and thus
$L_1L_2\subseteq L_1\cap L_2=\nl(B)\subseteq Q$, by condition~\eqref{item5} of Lemma~\ref{lem2}.
It follows that $L_j\subseteq Q$ for some $j$,
hence $L_jB_P\subseteq QB_P$, so 
$$QB_P\in V(L_jB_P)\cap \Spec^\circ(B_P)=V^\circ(L_jB_P)\subseteq V^\circ(L_1B_P)\cup V^\circ(L_2B_P)$$
as desired.

Claim 2: $V^\circ(L_1B_P)\cap V^\circ(L_2B_P)=\emptyset$.
By way of contradiction, suppose that $V^\circ(L_1B_P)\cap V^\circ(L_2B_P)\neq\emptyset$,
and let $QB_P\in V^\circ(L_1B_P)\cap V^\circ(L_2B_P)$.
It follows that $Q\subsetneq P$ and
$L_jB_P\subseteq QB_P$ for $j=1,2$.
It follows that
$$L_j\subseteq B\cap L_jB_P\subseteq B\cap QB_P=Q\subsetneq P$$
for $j=1,2$. (Given an ideal $I\leq B$, we use the notation $B\cap IB_P$ to denote the contraction of $IB_P$ along the 
natural map $B\to B_P$.)
From this, we have $L_1+L_2\subseteq Q\subsetneq P$, contradicting the fact that $P$ is minimal in $V(L_1+L_2)$.

To complete the proof, it suffices to show that $V^\circ(L_jB_P)\neq\emptyset$ for $j=1,2$.
For this, it suffices to show that there is a prime $P_{j,k}$ in $B$ such that $L_j\subseteq P_{j,k}\subsetneq P$.
Using condition~\eqref{item7} of Lemma~\ref{lem2}, we have
\begin{align*}
P_{j,1}\cap\cdots\cap P_{j,t_j}
&=L_j\subseteq L_1+L_2\subseteq P
\end{align*}
so we find that $P_{j,k}\subseteq P$ for some $k$.
By construction, we have 
$$L_j=P_{j,1}\cap\cdots\cap P_{j,t_j}\subseteq P_{j,k}\subseteq P.$$
Moreover, we have $\Ht(P_{j,k})=0<2=\Ht P$, so $P_{j,k}\subsetneq P$, as desired.
\end{proof}

{\it Conclusion of proof of Theorem \ref{thmtwo}.}  We have reduced the theorem to the case where $A$ is complete with separably closed residue field and contains a primitive $e$-th root of unity.  Set $R$ to be the truncated symbolic Rees algebra as in equation \eqref{eq121019a}.  It only remains to show that $\bar R=R/IR$ satisfies any of the equivalent conditions in Fact \ref{HH}.  The result will then follow from Theorem \ref{thmone}.

We will show that the $S_2$-ification of $\bar R$ is local.
To this end, note that $R$ is a local complete domain, hence equidimensional; see Lemma \ref{Rlocal}.  Set $n = \dim A = \dim R$, where the equality is from the fact that $R$ is finite over $A$.  

By way of contradiction, suppose the $S_2$-ification of $\bar R$ is not local.  Lemma~\ref{prop1} provides a prime ideal  $\bar P=P/IR\subset\bar R$ of height at least two such that $\Spec^{\circ}(\bar R_{\bar P})$ is disconnected. 
Now $R_P$ has a canonical module (since  it is a homomorphic image of a  regular local ring).  Moreover, it is $(S_2)$, therefore $\H_{PR_P}^{\dim R_P}(R_P)$ is indecomposable by \cite[Theorem 3.7]{HH}.  

\underline{Case 1.}  If $P$ is maximal in $R$, then $\bar R_{\bar P} = \bar R$, and $\mu(IR_P) \leq \mu(I) \leq n-2 = \dim R - 2 = \dim R_P - 2$, by condition (ii) of the hypotheses.  Thus, $\H_P^n(R)$ is indecomposable.  By \cite[Theorem 3.3]{HH}, $\Spec^{\circ}(\bar R_{\bar P})=\Spec^{\circ}(\bar R)$ is connected, contradicting our choice of $\bar P$.

\newcommand{\x}{\mathbf x}
\newcommand{\y}{\mathbf y}

\underline{Case 2.} If $P$ is not maximal in $R$, then $\frak p = P \cap A$ is not maximal in $A$ since $R$ is finite over $A$. Furthermore, we have $\frak p \supseteq IR \cap A \supseteq I$ since $P \supseteq IR$.   Therefore $I_{\fp}$ is generated by an $A_{\fp}$-regular sequence $\x=x_1,\ldots,x_c\in I_{\fp}$ by condition (iii) of the hypotheses.   

We claim that $\x$ is part of a system of parameters for $R_P$. (Then \cite[Theorem 3.9(c)]{HH} implies that $\Spec^{\circ}(\bar R_{\bar P})$ is connected, 
again contradicting our choice of $\bar P$.) Let $\x'=x_1,\ldots,x_c,\ldots,x_d\in \fp_{\fp}$ be a system of parameters for $A_{\p}$. 
It follows that $A_{\fp}/(\x')$ has finite length. Since $R_{\fp}$ is  finitely generated over $A_{\p}$, it follows that $R_{\fp}/(\x')$ has finite length,
that is, $R_{\fp}/(\x')$ is artinian. The ring $R_{P}/(\x')$ is a localization of $R_{\fp}/(\x')$, so it is also artinian.
Thus, the fact that the sequence $\x'\in P_P$ has length $d=\dim(A_{\fp})=\dim(R_P)$ by Lemma~\ref{AtoRetalecodimone}
implies that $\x'$ is a system of parameters for $R_P$, as claimed.

Thus, the $S_2$-ification of $\bar R$ is local.  The result now follows from Theorem \ref{thmone}.\qed

\section{Corollaries and an Example}
\label{sec130110b}
We begin this section by partially recovering~\cite[Theorem 1.2]{GW}, as  described in the discussion before Theorem \ref{thmtwo} in the introduction.

\begin{corollary}  \label{cor0401} Let $(A, \frak m)$ be an excellent, local, normal domain and $I$ a prime ideal of $A$ generated by an $A$-regular sequence such that $\bar A = A/I$ is normal.  For any integer $e > 1$ which represents a unit in $A$, the kernel of the homomorphism $\Cl(A) \to \Cl(\bar A)$ contains no element of order $e$.
\end{corollary}

\begin{proof} 
\underline{Case 1.} $\dim(A/I)\leq 1$. In this case, the fact that $A/I$ is normal (hence, it satisfies $(R_1)$) implies that $A/I$ is regular. Since $I$ is generated by an $A$-regular sequence, it follows from~\cite[Theorem 6.1(1)]{AFHa} that
$A$ is regular, so the result follows in this case. 

\underline{Case 2.} $\dim(A/I)\geq 2$. In this case, we have $\mu(I)=\dim(A)-\dim(A/I)\leq\dim(A)-2$ since $I$ is generated by an $A$-regular sequence.
Thus, the desired conclusion follows from Theorem \ref{thmtwo} in this case.
\end{proof}

The next two results follow from Theorem \ref{thmtwo}, as in \cite{GW}.

\begin{corollary} \label{cor0402}
(See \cite[Corollary 1.3]{GW}) Under the same assumptions as in Theorem \ref{thmtwo}, with the additional hypothesis that $A$ is $\mathbb Q$-algebra, the kernel of the homomorphism $\Cl(A) \to \Cl(\bar A)$ is torsion free.
\end{corollary}

\begin{corollary} \label{cor0403}
(See \cite[Corollary 1.4]{GW}) If the local hypothesis on $A$ is removed in Theorem \ref{thmtwo}, and instead it is assumed that $\operatorname{Pic}(A)=0$, and $I$ is a prime ideal of $A$ in the Jacobson radical of $A$ with finite projective dimension satisfying the three conditions of Theorem \ref{thmtwo}, then for any integer $e > 1$ representing a unit in $A$, the kernel of the homomorphism $\Cl(A) \to \Cl(\bar A)$ contains no element of order $e$.
\end{corollary}

We conclude with an example where the hypotheses of Theorem \ref{thmtwo} are satisfied by an ideal $I$ that is not generated by an $A$-regular sequence.

\begin{example}\label{ex120710a}
Let $(A, \fm)$ be an excellent normal 
\newcommand{\depth}{\operatorname{depth}}
local integral domain, and let the positive integer $e$ represent a unit in A. Assume that $A$ has an $A$-regular sequence $f_1,\ldots, f_6 \in \fm$
such that $A/(f_1,\ldots, f_6)A$ is a normal domain. (Examples of such rings are constructed in~\cite[Example 3.4]{SWS}.) In particular, we have $6 \leq \depth A \leq \dim A$. Localize at a minimal prime of the ideal $J = (f_1,\ldots, f_6)A$, if necessary, to assume that $J$ is $\fm$-primary. It follows that $A$ is Cohen-Macaulay of dimension 6. Arrange the sequence $f_1,\ldots, f_6$ in a $2\times 3$ matrix $F=\left(\begin{smallmatrix}f_1&f_2&f_3\\f_4&f_5&f_6\end{smallmatrix}\right)$, and consider the ideal $I=I_2(F)$ generated by the $2\times 2$ minors $d_1=f_1f_5-f_2f_4$, $d_2=f_2f_6-f_3f_5$, $d_3=f_1f_6-f_3f_4$. 

We claim that $A$ and $I$ satisfy the hypotheses of Theorem~\ref{thmtwo}. The proof of~\cite[Proposition 3.1]{SWS} shows that $I$ is a height-2 prime ideal of $A$ of finite projective dimension generated by $3$ elements (so $\mu(I)<\dim A-2$) such that $A/I$ is a normal domain. Thus, we need only show that $I$ is a complete intersection on the punctured spectrum. Let $\fp\in\Spec^{\circ}(A)$. We need to assume that $I\subseteq\fp$ and show that $IA_{\fp}$ is generated by an $A_{\fp}$-regular sequence.

Since $J$ is $\fm$-primary and $\fp\neq \fm$, we have $J\not\subseteq\fp$, so some generator of $J$ is not in $\fp$. By symmetry, assume that the generator $f_1$ is not in $\fp$. It follows that $f_1$ is a unit in $A_\fp$. Thus, a routine computation shows that in $A_{\fp}$ we have
$$d_2=f_2f_6-f_3f_5=\frac{f_2}{f_1}d_3-\frac{f_3}{f_1}d_1$$
so we conclude that $IA_{\p}=(d_1,d_3)A_{\p}$. Since $A_{\p}$ is Cohen-Macaulay and $\Ht(IA_{\fp})=2$, it follows that the sequence
$d_1,d_3$ is $A_{\p}$-regular (e.g., by~\cite[Theorem 17.4]{M}), as desired.
\end{example}

\section*{Acknowledgments}

The authors would like to thank Phillip Griffith for helpful conversations about \cite[Theorem 1.2]{GW}.


\providecommand{\bysame}{\leavevmode\hbox to3em{\hrulefill}\thinspace}
\providecommand{\MR}{\relax\ifhmode\unskip\space\fi MR }
\providecommand{\MRhref}[2]{%
  \href{http://www.ams.org/mathscinet-getitem?mr=#1}{#2}
}
\providecommand{\href}[2]{#2}

\end{document}